\documentclass[15pt]{article}

\usepackage[colorlinks,linkcolor=red,citecolor=blue]{hyperref}

\usepackage{amsmath,amsfonts,amssymb,amsthm,amscd}

\newcounter{stepctr}
{\end{list}}

\newtheorem{thm}{Theorem}[section]
\newtheorem{prop}[thm]{Proposition}
\newtheorem{cor}[thm]{Corollary}

\theoremstyle{definition}
\newtheorem{dfn}[thm]{Definition}
\newtheorem{ex}[thm]{Example}

\newtheorem{rema}[thm]{Remark}
\newtheorem{lem}[thm]{Lemma}
\newtheorem{prob*}{Open problem}
\newcommand{\demo}{\begin{proof}}

\newcommand{\R}{\ensuremath{\mathcal{R}}}

\newcommand{\N}{\mathbb{N}}

\newcommand{\C}{\mathbb{C}}

\def\ll^2{{\mathcal L}(\ell^2(\N))}

\def\f^0x{{\mathcal F^0}(X) }

\usepackage{color}
\usepackage{tikz}

\usetikzlibrary{shapes, arrows, calc, arrows.meta, fit, positioning} 
\tikzset{
    -Latex,auto,node distance =1.5 cm and 1.3 cm, thick,
    state/.style ={ellipse, draw, minimum width = 0.9 cm}, 
    point/.style = {circle, draw, inner sep=0.18cm, fill, node contents={}},
    bidirected/.style={Latex-Latex,dashed}, 
    el/.style = {inner sep=2.5pt, align=right, sloped}
}

\title
{\bf  On subclasses   of   Browder and Weyl operators}

\author{   Z. Aznay, A. Ouahab, H. Zariouh }

\date{}
\begin{document}

\maketitle \thispagestyle{empty}

\begin{abstract}\noindent\baselineskip=10pt
 The main purpose of this paper, is to  introduce and study the classes $(ab_{e})$ and $(aw_{e})$ which are
strongly related to what has been recently studied in \cite{aznay-zariouh}. Furthermore, we give the connection between these classes and those that have been studied in \cite{berkani-zariouh0}.   We also give an affirmative answer to a question asked in \cite{aznay-zariouh}.
\end{abstract}

 \baselineskip=15pt
  \baselineskip=15pt

\section{Introduction}

In 1909, H.Weyl \cite{W} examined the spectra of all compact perturbation of
 a self-adjoint operator on a Hilbert space, and found that their intersection consisted precisely of
  those points of the spectrum which were not isolated eigenvalues of finite multiplicity. Today this
   classical result may be stated by saying that the spectral points of a self-adjoint operator which
   do not belong to Weyl spectrum are precisely the eigenvalues of finite multiplicity which are isolated points
    of the spectrum. This Weyl's theorem has been extended from self-adjoint operators to  normal operators and  to several other Banach spaces classes of
    operators, and many new variants such as    Rako\v{c}evi\'{c}'s, Browder's
 have been obtained by many researchers, see for example \cite{harte-lee, berkani-zariouh1,  berkani-zariouh0, gupta-kashyap, rako1, rako2}.
Moreover, in \cite{Berkani-koliha} Berkani extended   Weyl's  theorem (and some their variants) in the context of B-Fredholm theory to a new variant called generalized Weyl's theorem and he proved in particular that  the class of normal operators $(\mathcal{N})$  is a   subclass of the class $(gW)$  of operators satisfying  generalized Weyl's theorem, which in turn is a   subclass of the class $(W)$ of operators satisfying Weyl's theorem; so that $(\mathcal{N})\subset (gW)\subset (W).$ On the other hand, a result due to Conway \cite[Chapter XI, Proposition 4.6]{conway} shows that  the essential spectrum  of a normal operator    consists precisely of all points in its spectrum except the isolated eigenvalues of finite multiplicity. As a motivation for this result, we went very recently (see \cite{aznay-zariouh}) in the same direction by determining the class named  $(W_{e})$ of  operators which satisfy this result  and we proved that $(\mathcal{N})\subset(gW_{e})=(gW)\cap(W_{e})\subset(W);$ where $(gW_{e})$ is the generalization class of $(W_{e})$ in  the context of B-Fredholm theory. As a continuation of \cite{aznay-zariouh}, we extend these result by studying others new classes of operators named $(gaw_{e})$ and  $(aw_{e})$ and we prove that $(\mathcal{N})\subset (gaw_{e})=(gaw)\cap (gW_{e})\subset (aw_{e})=(aw)\cap (W_{e});$ where $(gaw)$ and  $(aw)$ are the classes studied in \cite{berkani-zariouh0}. Furthermore, we give an affirmative answer to a question asked  in \cite{aznay-zariouh}.

\section{Terminology and preliminaries}

Let $X$ denote an infinite dimensional complex Banach space, and
denote by $L(X)$ the algebra of all bounded linear operators on $X.$
For $T\in L(X),$ we  denote by $T^*,$ $\alpha(T)$  and   $\beta(T)$ the dual of $T,$ the dimension of the
kernel $\mathcal{N}(T)$  and  the codimension of the range $\R(T),$ respectively.
By $\sigma (T)$ and  $\sigma_a(T),$  we denote the spectrum and
 the approximate spectrum of $T,$
respectively.\\
 Recall that $T$ is said to be upper semi-Fredholm, if $\R(T)$ is
closed and $\alpha(T) <\infty,$ while $T$ is called lower
semi-Fredholm, if   $\beta(T) < \infty.$ $T\in
L(X)$ is said to be semi-Fredholm if $T$ is either an upper
semi-Fredholm or a lower semi-Fredholm operator. $T$ is
Fredholm if $T$ is upper semi-Fredholm and lower
semi-Fredholm. If $T$ is semi-Fredholm then the index of $T$ is
defined by $\mbox{ind}(T)=\alpha(T) -\beta(T).$ For an operator $T \in
L(X),$ the ascent $p(T)$ and the descent $q(T)$
are defined by $p(T) = \inf\{n\in \mathbb{N}: \mathcal{N}(T^n) = \mathcal{N}(T^{n+1})\}$
 and $q(T)= \inf\{n\in \mathbb{N}: \R(T^n) = \R(T^{n+1})\},$
 respectively; the
infimum over the empty set is taken $\infty.$  An operator $T$ is said to be Weyl if it is
Fredholm of index zero and is said to be Browder if it is Fredholm of finite
ascent and descent. \\
If $T\in L(X)$ and $ n \in \mathbb{N},$
 we denote  by $T_{[n]}$ the restriction
of  $T$ on $\R(T^n).$ $T$ is said to be semi-B-Fredholm  if there
exists $n\in \mathbb{N}$ such that
$\R(T^n) $ is closed and $T_{[n]}:
\R(T^n)\rightarrow \R(T^n)$ is semi-Fredholm. A B-Fredholm (resp., B-Weyl) operator is similarly defined.

 We recall that a complex number $\lambda\in\sigma (T)$ is a
pole of the resolvent of $T,$ if $T-\lambda I$ has finite
ascent and finite descent, and $\lambda\in\sigma_a(T)$ is a
left pole of $T$ if $p=p(T-\lambda I) <\infty$ and
$\R((T-\lambda I)^{p+1})$ is closed.
In the following  list, we summarize  the notations and symbols
needed later.

\smallskip
\noindent $\mbox{iso}\,A$: isolated points in a given subset A of $ \mathbb{C},$
$\mbox{acc}\,A$: accumulation points of a given subset A of $ \mathbb{C},$
  $A^C$: the complementary of a subset $A\subset \mathbb{C}.$ For $\lambda \in \C,$ we denote by
 $B(\lambda, \epsilon)=\{\mu \in \C : |\lambda -\mu|<\epsilon\}$ the  open ball of radius $\epsilon$ centred at $\lambda,$
 $D(\lambda, \epsilon)=\{\mu \in \C : |\lambda -\mu|\leq \epsilon\}$ the  closed  ball of radius $\epsilon$ centred at $\lambda$ and
$C(\lambda, \epsilon)=\{\mu \in \C : |\lambda -\mu|= \epsilon\}$ the  circle of radius $\epsilon$ centred at $\lambda.$\\
\noindent $(B)$: the class of operators satisfying Browder's theorem [$T\in(B)$ if $\Delta(T)=\Pi^0(T)$],\\
\noindent $(gB)$: the class of operators satisfying generalized Browder's theorem [$T\in(gB)$ if $\Delta^g(T)=\Pi(T)$],\\
 \noindent $(W)$: the class of operators satisfying Weyl's theorem [$T\in(W)$ if $\Delta(T)=E^0(T)$],\\
\noindent $(gW)$: the class of operators satisfying generalized Weyl's theorem [$T\in(gW)$ if $\Delta^g(T)=E(T)$].

\vspace{10pt}
 \begin{tabular}{l|l}
 $\sigma_{b}(T)$:  Browder spectrum of $T$ &   $\Delta^g(T):=\sigma(T)\setminus\sigma_{bw}(T)$\\
 $\sigma_{ub}(T)$:  upper semi-Browder spectrum of $T$ &$\Delta_{e}(T):=\sigma(T)\setminus\sigma_{e}(T)$\\
 $\sigma_{w}(T)$: Weyl spectrum of $T$ & $\Delta_{e}^g(T):=\sigma(T)\setminus\sigma_{bf}(T)$\\
 $\sigma_{uw}(T)$: upper semi-Weyl spectrum of $T$ &  $\Pi^0(T)$: poles of $T$ of finite rank\\
 $\sigma_{e}(T)$: essential spectrum of $T$ & $\Pi(T)$: poles of $T$\\
 $ \sigma_{d}(T)$:  Drazin spectrum of $T$ &  $\Pi_a^0(T)$: left  poles of $T$ of finite rank \\
  $\sigma_{ld}(T)$: left Drazin spectrum of $T$ & $\Pi_a(T)$: left  poles of $T$\\
$\sigma_{bw}(T)$: B-Weyl spectrum of $T$ &  $E^0(T):=\mbox{iso}\,\sigma(T)\cap\sigma_{p}^0(T)$ \\
 $\sigma_{bf}(T)$: B-Fredholm  spectrum of $T$ &  $E(T):=\mbox{iso}\,\sigma(T)\cap\sigma_{p}(T)$\\
 $\sigma_{p}(T)$:  eigenvalues of $T$ & $E_a^0(T):=\mbox{iso}\,\sigma_{a}(T)\cap\sigma_{p}^0(T)$\\
    $\sigma_{p}^0(T)$: eigenvalues of $T$ of finite multiplicity & $E_a(T):=\mbox{iso}\,\sigma_{a}(T)\cap\sigma_{p}(T)$\\
 $\Delta(T):=\sigma(T)\setminus\sigma_{w}(T)$ &  \\
\end{tabular}\\

For more details on several classes and spectra originating from Fredholm  or B-Fredholm theory, we refer the reader to \cite{aiena1, Berkani-koliha}.



\section{  The  class  $(ab_{e})$-operators}

 In \cite[Definition 2.1]{aznay-zariouh}, we  have introduced and studied the class $(B_{e})$-operators,  as a  subclass of the class $(B)$ of operators satisfying Browder's theorem.  As a continuation of \cite{aznay-zariouh, berkani-zariouh0},   we  introduce in the next definition a new class of operators named $(ab_{e})$-operators,  as a  subclass of the classes $(B_{e})$ and $(ab).$ Recall \cite{aznay-zariouh, berkani-zariouh0} that  $T\in (ab)$ if $\Delta(T)=\Pi_{a}^0(T)$ and that  $T\in (B_{e})$ if  $\Delta_{e}(T)=\Pi^0(T)$ [or equivalently $\sigma_{e}(T)=\sigma_{b}(T)].$

\begin{dfn}\label{dfn0} A bounded linear operator $T\in L(X)$ is said to belong to the class $(ab_{e})$ [$T\in (ab_{e})$ for brevity] if  $\Delta_{e}(T)=\Pi_a^0(T),$  and is said to belong to the class
 $(gab_{e})$ [$T\in (gab_{e})$ for brevity] if $\Delta_{e}^g(T)=\Pi_a(T).$
\end{dfn}

\begin{ex}\label{ex0}

\noindent (1) Let    $T$ be the operator  defined on $l^2(\mathbb{N})$ (which is usually denoted by $\l^2$) by $T(x_1,x_2,x_3,\ldots)=(x_1,\frac{x_2}{2}, \frac{x_3}{3},\ldots).$ It is  easily seen that  $\sigma(T)=\sigma_a(T)= \{0\}\cup \{\frac{1}{n}\}_{n\geq1}$  and  $\sigma_{bf}(T)=\sigma_{e}(T)=\sigma_{ub}(T)=\sigma_{ld}(T)=\{0\}.$ So  $T\in (ab_{e})\cap (gab_{e}).$

\noindent (2)  Let $Q$ be the operator  defined on $l^1(\mathbb{N})$ (which is usually denoted by $\l^1$) by
$Q(x_1,x_2,...)=(0,\alpha_{1}x_1, \alpha_{2}x_2,\ldots, \alpha{_k}x_k, \ldots),$ where $(\alpha_{i})$ is a sequence of complex numbers such that $0<|\alpha_{i}|< 1$ and $\sum_{i=1}^{\infty} |\alpha_{i}|<\infty,$
and we define the operator  $T$ on $l^1\oplus l^1$ by $T=Q\oplus 0.$ We have $\sigma(T)=\sigma_{a}(T)=\{0\}.$  It is easily seen that the range   $\mathcal{R}(T^n)$ of $T^n$ is not closed for any $n\in\mathbb{N},$ so that   $\sigma_{bf}(T)=\sigma_{e}(T)=\{0\}$ and $\Pi_{a}^0(T)=\Pi_{a}(T)=\emptyset.$ Thus  $T\in (ab_{e})\cap (gab_{e}).$
\end{ex}


Our next  proposition  prove that   the class $(ab_{e})$ contains the class $(gab_{e}).$  Moreover, the  example given below  shows that this inclusion in general  is proper.

\begin{prop}\label{prop0} Let $T\in L(X).$   If  $T \in (gab_{e}),$ then $T \in (ab_{e}).$
\end{prop}
\begin{proof}   Suppose that $T\in (gab_{e})$ that's $\Delta_{e}^g(T)=\Pi_a(T).$  As $\Delta_{e}(T)\subset\Delta_{e}^g(T),$ then $\Delta_{e}(T) \subset \Pi_a^0(T).$ Conversely,  let $\lambda \in \Pi_{a}^0(T).$ As  $T \in (gab_{e}),$ then $\lambda \in [\sigma_{ub}(T)\cup \sigma_{bf}(T)]^C,$ and
 so  $T - \lambda I$   is both a semi-Fredholm  and a B-Fredholm operator. Hence  $T - \lambda I$ is  a Fredholm operator and $\lambda \in \Delta_{e}(T).$ Thus  $T \in (ab_{e}).$
\end{proof}

\begin{ex}\label{ex1}
Let $S$ be the operator  defined on $l^2$ by $S(x_1,x_2,x_3,\ldots)=(x_1,0, x_2,0, x_3, \ldots).$ And hereafter, let  $R$ denote the unilateral  right  shift operator  defined on $l^2$ by $R(x_1,x_2,x_3,\ldots)=(0, x_1, x_2, x_3, \ldots).$ We have $\sigma(S)=\sigma_{e}(S)=\sigma_{bf}(S)=D(0,1)$ and $\sigma_{uf}(S)=\sigma_{a}(S)=C(0,1).$   We consider  the operator   $T=S\oplus R \oplus 0;$ where $0$ means the null operator acting on $l^2.$   It is easily seen that   $\sigma(T)=\sigma_{e}(T)=\sigma_{bf}(T)=D(0, 1),$ $\sigma_{ub}(T)=\sigma_{a}(T)=C(0, 1)\cup \{0\}$ and $\sigma_{ld}(T)=C(0, 1).$ So $\Delta_{e}(T)=\Delta_{e}^g(T)=\Pi_a^0(T)=\emptyset$ and  $\Pi_a(T)=\{0\}.$  This proves that  $T \in (ab_{e})$ and   $T \not\in (gab_{e}).$
\end{ex}

\begin{prop}\label{newprop0}
Every Riesz operator $T \in L(X)$ belongs to the class  $(gab_{e}).$
\end{prop}
\begin{proof}
As $T$ is a Riesz operator, then $\sigma(T)=\sigma_{a}(T).$ And since  $\sigma_{e}(T)=\sigma_{ub}(T)=\{0\},$ from \cite[Theorem 2.14]{aznay-zariouh} we have    $\sigma_{bf}(T)=\sigma_{ld}(T).$ Hence $T \in (gab_{e}).$
\end{proof}

\begin{rema}\label{newrema0}
The classes $(ab_{e})$ and $(gab_{e})$ are not stable under the duality. To see this,
 we consider the operator $A=P\oplus 2S;$ where $S$ is the operator defined  in  the Example \ref{ex1} and $P$ is the projection defined on $l^2$ onto $\mathbb{C}e_{1};$ where $e_1=(1, 0, 0, \ldots).$ Then $\sigma(A)=\sigma_{bf}(A)=D(0, 2),$  $\sigma_{ub}(A)=\{0\} \cup C(0, 2)$ and $\sigma_{a}(A)=\{0, 1\} \cup C(0, 2).$ Hence     $A\notin (ab_{e})$ and then $A\notin (gab_{e}).$ However $A^*\in (gab_{e}),$ since $\sigma_{a}(A^*)=D(0, 2).$

\end{rema}
For $T\in L(X),$ we have always   $\Delta(T)\subset \sigma_{a}(T)$ and $\Delta^g(T)\subset \sigma_{a}(T).$   However, we cannot guarantee these two inclusions, if we replace $\Delta(T)$ by $\Delta_{e}(T)$ and $\Delta^g(T)$ by $\Delta_{e}^g(T).$ As example, we have      $\Delta_{e}(R)=\Delta_{e}^g(R)=B(0, 1)\not\subset \sigma_{a}(R)=C(0,1).$ Nonetheless, we have the following lemma which will be useful in the sequel. For definitions and properties of operators with topological uniform descent, we refer the reader to the paper of Grabiner \cite{Grabiner}.

\begin{lem}\label{newlem0}  Let $T \in L(X).$  Then the following statements hold. \\
(i) $\Delta_{e}(T)\subset \sigma_{a}(T)$ $\Longleftrightarrow$ $\Delta_{e}^g(T)\subset \sigma_{a}(T).$\\
(ii) $\sigma_{e}(T)\subset \sigma_{a}(T)$ $\Longleftrightarrow$  $\sigma_{bf}(T)\subset \sigma_{a}(T).$\\
(iii) If $\Pi_{a}(T)\subset\Delta_{e}^g(T)$ then $\Pi_{a}^0(T)\subset\Delta_{e}(T).$\\
(iv) If $\Pi_{a}(T)\subset\Delta^g(T)$ then $\Pi_{a}^0(T)=\Pi^{0}(T).$
\end{lem}
\begin{proof}
(i)``If" let $\lambda \in \Delta_{e}^g(T),$  then by \cite[Corollary 4.8]{Grabiner} there exists $\epsilon >0$ such that $B_{\lambda}:=B(\lambda, \epsilon)\setminus\{\lambda\}\subset(\sigma_{e}(T))^C.$ If   $B_{\lambda}\cap(\sigma_{a}(T))^C\neq \emptyset,$ then by hypotheses there exists $\mu \not\in \Delta_{e}(T)$ such that  $\mu \notin \sigma_{e}(T).$ So $\mu \notin \sigma(T).$  We deduce from  \cite[Corollary 4.8]{Grabiner},  that  $\lambda \in \Pi(T)$ and then  $\lambda \in \sigma_{a}(T).$
If   $B_{\lambda}\cap(\sigma_{a}(T))^C= \emptyset,$ then for all $\mu \in B_{\lambda},$ we have $p(T-\mu I)=p(T- \lambda I)=\infty.$ Hence $\lambda \in \sigma_{a}(T)$ and then   $\Delta_{e}^g(T)\subset \sigma_{a}(T).$   ``Only if" obvious, since $\Delta_{e}(T)\subset \Delta_{e}^g(T).$ \\
(ii) ``If" is obvious. ``Only if"  let $\lambda \notin \sigma_{a}(T),$ then $\lambda \notin \sigma_{bf}(T).$  So that $T - \lambda I$   is  a semi-Fredholm  and a B-Fredholm operator. So $\lambda \notin \sigma_{e}(T).$\\
(iii) Since $\Pi_{a}^0(T)\subset \Pi_{a}(T),$ then  $\Pi_{a}^0(T)\subset\Delta_{e}^g(T).$ Let $\lambda \in  \Pi_{a}^0(T),$ then $T - \lambda I$ is  a semi-Fredholm and a B-Fredholm operator.  Hence    $\lambda \in \Delta_{e}(T).$\\
(iv) Goes similarly with  (iii).
\end{proof}

Note that  since  from \cite{aznay-zariouh},  the class  $(B_{e})$ and its generalization class $ (gB_{e})$ in the context of B-Freholm theory coincide, then we only refer to the class $ (B_{e})$ when necessary. We also recall the following equivalence which will be useful in the sequel: $\sigma_{e}(T)=\sigma_{w}(T)\Longleftrightarrow \sigma_{bf}(T)=\sigma_{bw}(T).$
\begin{thm}\label{lem0} Let $T \in L(X).$  The following statements hold.\\
(i) If   $\Pi_{a}^0(T)\subset\Delta_{e}(T)\subset\sigma_{a}(T),$  then $\Pi^0(T)=\Pi_{a}^0(T).$\\
(ii) If   $\Pi_{a}(T)\subset\Delta_{e}^g(T)\subset\sigma_{a}(T),$  then $\Pi(T)=\Pi_{a}(T)$ and $\Pi^0(T)=\Pi_{a}^0(T).$ \\
(iii) $T \in (B_{e})$ $\Longleftrightarrow$ $\Delta_{e}(T)\subset\Pi(T)$ $\Longleftrightarrow$ $\Delta_{e}(T)\subset\Pi_{a}^0(T)$ $\Longleftrightarrow$ $\Delta_{e}(T)\subset\Pi_{a}(T)$ $\Longleftrightarrow$ $\Delta_{e}^g(T)\subset\Pi_{a}(T).$
\end{thm}
\begin{proof}
(i) Let $\lambda \in \Pi_{a}^0(T),$  then $T-\lambda I$ is an operator of topological uniform descent for   $n\geq d;$  where $d$ is the degree of stable iteration of $T-\lambda I.$ As $(T-\lambda I)^{d+1}$ is also a semi-Fredholm operator, then $\epsilon:=\gamma((T-\lambda I)_{|\R(T-\lambda I)^d})>0.$    For  $\mu \in  B_{\lambda}:=B(\lambda, \epsilon)\setminus\{\lambda\},$ we  set $S=T-\lambda I$ and $V=T-\mu I.$ Then $\|(V-S)_{|\R(T-\lambda I)^d}\|=\|(\lambda-\mu)I_{|\R(T-\lambda I)^d}\|=|\lambda-\mu|<\epsilon.$  Hence the operator $V-S$ is  sufficiently small and invertible.  As $\Pi_{a}^0(T) \subset\Delta_{e}(T)$ then by \cite[Corollary 4.8] {Grabiner},  $V$ is a Fredholm operator with $\alpha(V)=0.$  Hence $ B_{\lambda} \subset (\sigma_{e}(T)\cup\sigma_{a}(T))^C.$  Let  us to  show  that  $B_{\lambda}\subset(\sigma(T))^C.$  Suppose to the contrary, that's there exists $\mu \in B_{\lambda}$ such that $\mu \in \sigma(T).$ Then $\mu \in \Delta_{e}(T)\subset\sigma_{a}(T),$   and this is a contradiction.     Using again  \cite[Corollary 4.8] {Grabiner}, we deduce that  $\lambda \in \Pi^0(T).$  Hence $\Pi^0(T)=\Pi_{a}^0(T).$\\
\noindent (ii) Goes similarly with the first point. Remark that  Lemma \ref{newlem0} proves that the   condition $\Pi_{a}(T)\subset\Delta_{e}^g(T)\subset\sigma_{a}(T)$  implies  that $\Pi_{a}^0(T)\subset\Delta_{e}(T)\subset\sigma_{a}(T).$\\
\noindent (iii) Let us to show only that  $\Delta_{e}(T)\subset\Pi_{a}(T) \Longrightarrow \Delta_{e}^g(T)\subset\Pi_{a}(T).$ All other equivalences are clear and are left to the reader. From Lemma \ref{newlem0}, we have  $\Delta_{e}^g(T)\subset \sigma_{a}(T).$ Let $\lambda \in \Delta_{e}^g(T)$ be arbitrary.  From the punctured neighborhood theorem for semi-B-Fredholm operators,  there exists $\epsilon >0$ such that $B_{\lambda}:=B(\lambda, \epsilon)\setminus\{\lambda\}\subset(\sigma_{e}(T))^C.$ If there exists $\mu \in B_{\lambda}\cap \sigma_{a}(T),$  then $\mu \in \Delta_{e}(T)\subset\Pi_{a}(T)$ and from the Grabiner's theory, we deduce that $\mu\not\in\sigma_{a}(T),$   this is a contradiction. So $B_{\lambda}\subset (\sigma_{a}(T))^C.$ Hence $\lambda \notin \sigma_{ld}(T)$ and then $\lambda \in \Pi_{a}(T).$
\end{proof}

The condition     ``$\Delta_{e}(T)\subset\sigma_{a}(T)$"  (which is  equivalent by Lemma \ref{newlem0} to ``$\Delta_{e}^g(T)\subset\sigma_{a}(T)$") assumed in the previous theorem is crucial. Indeed,  we consider the  operator $U=T\oplus R;$ where $T$ is the operator defined in the first point of  Example \ref{ex0}.  Then $\sigma(U)=\sigma_{w}(U)=\sigma_{b}(U)=D(0, 1),$  $\sigma_{a}(U)=\{0\}\cup\{1/n\}_{n\geq 2} \cup C(0, 1)$ and  $\sigma_{e}(U)=\sigma_{bf}(U)=\sigma_{ub}(U)=\sigma_{ld}(U)=\{0\} \cup C(0, 1).$ So $\Pi_{a}^0(U)=\Pi_{a}(U)=\{1/n\}_{n\geq 2}\subset D(0, 1)\setminus \left(\{0\}\cup C(0, 1)\right)=\Delta_{e}(U)=\Delta_{e}^g(U).$ But $\Pi^0(U)=\Pi(U)=\emptyset\neq\Pi_{a}^0(U).$

 \vspace{5pt}
From Theorem \ref{lem0} we obtain immediately the following corollary, which gives a relationship between the classes  $(ab_{e}),$   $(gab_{e})$ and the class $(B_{e}).$
\begin{cor}\label{cor0} Let  $T\in L(X).$ The following statements  hold.\\
  (i) $T\in (ab_{e})$  if and only if  $T \in (B_{e})$ and $\Pi^0(T)=\Pi_a^0(T).$\\
  (ii) $T\in (gab_{e})$  if and only if  $T \in (B_{e})$ and $\Pi(T)=\Pi_a(T).$
\end{cor}
  Corollary \ref{cor0}  above shows that the class $(ab_{e})$ is included in the class $(B_{e}).$   However, the   operator  given in Remark \ref{newrema0} above, shows that this inclusion in general is proper.

 \vspace{5pt}
Our next proposition gives a relationship between the class $(ab_{e})$ (resp., the class $(gab_{e})$) and the class   $(ab)$ (resp., the class  $(gab)$). Its proof is simple and is left to the reader.  Recall that \cite{berkani-zariouh0}, $T\in (gab)$ if $\Delta^g(T)=\Pi_{a}(T).$ And it is proved in \cite{berkani-zariouh0} that an operator $T\in (gab)$  implies that  $T\in (ab),$ but in  general not  conversely.

\begin{prop}\label{cor1}  Let $T\in L(X).$  The following statements  hold.\\
(i)  $T\in (ab_{e})$   if and only if   $T \in (ab)$     and $\sigma_{e}(T)=\sigma_{w}(T).$\\
(ii) $T\in (gab_{e})$  if and only if   $T \in (gab)$     and $\sigma_{e}(T)=\sigma_{w}(T).$
\end{prop}
From Corollary \ref{cor0} and Proposition \ref{cor1}, we conclude   that $(ab_{e})=(ab)\cap (B_{e})$ and  $(gab_{e})=(gab)\cap(B_{e}).$
Moreover, we show by the following example that the inclusions  $(ab_{e})\subset(ab)$ and $(gab_{e})\subset(gab)$   in general are proper. Here and elsewhere,
$L$ denotes  the unilateral left shift operator defined on  $l^2$ by $L(x_1, x_2, \dots)=(x_2, x_3, \dots).$  It is easily seen    that $\sigma(L)=\sigma_{w}(L)=\sigma_{a}(L)=\sigma_{ub}(L)=\sigma_{ld}(L)=D(0, 1)$ and  $\sigma_{e}(L)=C(0, 1).$
So $\Delta^g(L)=\Pi_a(L)=\emptyset \neq  \Delta_{e}(L).$ Thus  $L\in  (gab)$ and then  $L\in  (ab).$   But $L \not\in (ab_{e})$ and then $L \not\in (gab_{e}).$

\begin{rema}\label{rema0} It follows from  Corollary \ref{cor0} and Proposition \ref{cor1} that   if $T\in  (ab_{e}),$ then $\Delta(T)=\Delta_e(T)=\Pi^0(T)=\Pi_a^0(T)$ and $\Delta^g(T)=\Delta_e^g(T)=\Pi(T).$  And if in addition,  $T\in (gab_{e})$ then $\Delta(T)=\Delta_e(T)=\Pi^0(T)=\Pi_a^0(T)$ and $\Delta^g(T)=\Delta_e^g(T)=\Pi(T)=\Pi_a(T).$
\end{rema}
From Proposition \ref{prop0} and Remark \ref{rema0},  we obtain the following result which gives  a sufficient condition for an operator belonging to the  class  $(ab_{e}),$ to belong to the  class $(gab_{e}).$
\begin{cor}\label{cor2} Let  $T \in L(X).$ Then
 $T \in (gab_{e})$ if and only if
 $T \in (ab_{e})$ and $\Pi(T)=\Pi_a(T).$
\end{cor}

\par In \cite{aznay-zariouh}, we had mentioned  that the left shift operator   $L \in (aB),$ but  $L \notin (B_{e})$ and  we had asked   the following question. Does there exist an operator $T\in L(X)$ such that $T\in(B_{e}),$  but $T\not\in (aB)$?
 We   answer here  this question  affirmatively. Note that  $(aB)$ is the class of operators satisfying the classical a-Browder's theorem that is $T\in(aB)$ if $\sigma_{uw}(T)=\sigma_{ub}(T).$
\begin{ex}\label{ex4}~~
We consider the operator $T$ defined on    $l^2\oplus l^2$ by $T=L\oplus S,$ where  $S(x_1,x_2,x_3,\ldots)=(x_1,0, x_2,0, x_3, \ldots)$ is the operator defined in Example \ref{ex1}. We have   $\sigma(T)=\sigma_{e}(T)=\sigma_{ub}(T)=D(0, 1)$  and $\sigma_{ue}(T)=\sigma_{uw}(T)=C(0, 1).$ Thus  $T\in (B_{e}),$ but $T\not \in (aB).$ So the classes $(aB)$ and  $(B_{e})$ are  independent.
\end{ex}

\section{ The class  $(aw_{e})$-operators}
According to  \cite{berkani-zariouh0}, we say that $T\in (aw)$ if $\Delta(T)=E_{a}^0(T)$ and according to \cite{aznay-zariouh}, we say that $T\in (W_{e})$ if $\Delta_{e}(T)=E^{0}(T).$  As a continuation of \cite{aznay-zariouh, berkani-zariouh0} and the first part of this paper,  we introduce   a new class called   $(aw_{e})$-operators,   as a  subclass of the classes $(ab_{e})$ and $(aw).$ Furthermore, we    prove that $(aw_{e})=(aw)\cap(W_{e}).$

\begin{dfn}\label{dfn1}A bounded linear operator $T\in L(X)$ is said to belong to the class $(aw_{e})$ [$T\in (aw_{e})$ for brevity] if  $\Delta_{e}(T)=E_{a}^0(T),$ and is said to belong to the class
    $(gaw_{e})$ [$T\in (gaw_{e})$ for brevity] if $\Delta_{e}^g(T)=E_{a}(T).$
\end{dfn}

\begin{ex}\label{ex5}
\noindent (1)  Every normal  operator  $N$ acting on a Hilbert space belongs to  $(gaw_{e})\cap(aw_{e}).$ Indeed, from \cite[Corollary 3.10]{aznay-zariouh} we obtain $\Delta_{e}(N)=E^0(N)=E_{a}^0(N)=\{\mbox{iso}\,\sigma(N) : \alpha(N-\lambda I)<\infty\}$  and $\Delta_{e}^g(N)=E(N)=E_{a}(N)=\mbox{iso}\,\sigma(N).$\\
 \noindent (2) Let $T$ be the operator  defined on $l^2$   by  $T(x_1,x_2,x_3,\ldots)=(\frac{x_2}{2},\frac{x_3}{3}, \frac{x_4}{4},\ldots).$ Since  $E_{a}(T)=E_{a}^0(T)=\{0\}$ and $\sigma(T)=\sigma_{e}(T)=\sigma_{bf}(T)=\{0\},$ then $T\not\in (aw_{e})\cup(gaw_{e}).$\\
 \noindent (3) Every injective quasinilpotent operator $T$ acting on an infinite dimensional Banach space belongs to   $(gaw_{e})\cap(aw_{e}),$ since $\sigma_{bf}(T)=\sigma_{e}(T)=\{0\}=\sigma(T)$ and $E_{a}(T)=E_{a}^0(T)=\emptyset.$
 \end{ex}

The following proposition, gives  a relationship between the class    $(aw_{e})$ [resp., $(gaw_{e})$] and  the class $(ab_{e})$ [resp., $(gab_{e})$].

\begin{prop}\label{prop1} Let $T\in L(X).$ The following statements hold.\\
(i) $T\in (aw_{e})$ if and only if $T\in (ab_{e})$ and $\Pi_{a}^0(T)=E_{a}^0(T).$\\
(ii)  $T\in (gaw_{e})$ if and only if $T\in (gab_{e})$ and  $\Pi_{a}(T)=E_{a}(T).$
\end{prop}
\begin{proof}
(i) Assume that   $T\in (aw_{e}).$ As $\Pi_{a}^0(T)\subset E_{a}^0(T),$ then $\Pi_{a}^0(T)\subset\Delta_{e}(T).$ Moreover, $\Delta_{e}(T)=\Delta_{e}(T)\cap E_{a}^0(T)\subset \mbox{iso}\,\sigma_{a}(T)\cap(\sigma_{e}(T))^C.$ Since  the inclusion $\mbox{iso}\,\sigma_{a}(T)\cap(\sigma_{e}(T))^C\subset \Pi_{a}^0(T)$ is always true, it follows that   $T\in (ab_{e})$ and $\Pi_{a}^0(T)=E_{a}^0(T).$ The converse is obvious.\\
(ii) Goes similarly with the first point.
\end{proof}

Remark that the operator $T$ given in the second point of  Example \ref{ex5}, shows that the class $(aw_{e})$  in general is a proper subclass of the class $(ab_{e}).$ It shows also that the class  $(gaw_{e})$ in general is a proper subclass of  the class $(gab_{e}).$

From Proposition \ref{prop1} and Corollary \ref{cor2}, we obtain the following corollary, which gives  a relationship between the classes $(aw_{e})$ and  $(gaw_{e}).$

\begin{cor}\label{cor3} Let $T\in L(X).$ Then
   $T\in (gaw_{e})$ if and only if  $T\in (aw_{e})$ and $E_{a}(T)=\Pi(T).$
\end{cor}
\begin{proof}
 Assume that $T\in (gaw_{e}).$ From Proposition \ref{prop1} and Corollary \ref{cor2},   we have $T\in (gab_{e})$ and   $E_{a}(T)=\Pi(T)=\Pi_a(T).$ So $E_{a}^0(T)= \Pi_{a}^0(T).$  Hence $T\in (aw_{e})$ and $E_{a}(T)=\Pi(T).$   Conversely,  suppose that  $T\in (aw_{e})$ and $E_{a}(T)=\Pi(T).$ So $E_{a}(T)=\Pi(T)=\Pi_{a}(T).$  From  Proposition \ref{prop1},   $T\in (ab_{e})$ and  by Remark \ref{rema0},  $\Delta_{e}^g(T)=\Pi(T)=E_{a}(T).$ Thus  $T \in (gaw_{e}).$
\end{proof}
 From Corollary \ref{cor3},   the class  $(gaw_{e})$ is included   in the class
 $(aw_{e}).$ Furthermore, this inclusion in general  is proper.  For this, we consider    the operator $B$ defined on $l^2\oplus l^2$ by $B=Q\oplus 0;$ where $Q$ is any injective quasi-nilpotent  operator defined on an infinite dimensional Banach space. We have   $B\in (aw_{e}),$ since $\sigma(B)=\sigma_{e}(B)=\{0\}$ and $E_{a}^0(B)=\emptyset.$ But  $B\not\in(gaw_{e}),$ since $\sigma_{bf}(B)=E_a(B)=\{0\}.$ Here $\Pi(B)=\emptyset.$

\begin{rema}\label{rema2} The classes $(aw_{e})$ and $(gaw_{e})$ are not stable under the duality. To see this,  let $S$ be  the operator defined on $l^2$ by  $S(x_1,x_2,x_3,\ldots)=(0, \frac{x_1}{2},\frac{x_2}{3}, \frac{x_3}{4},\ldots).$ Since $S$ is quasi-nilpotent and compact operator which is not of  finite rank,  then  $\sigma(S)=\sigma_{e}(S)=\sigma_{bf}(S)=\{0\}$  and $E_{a}(S)=E_{a}^0(S)=\emptyset.$  So $S\in (gaw_{e}).$
   But its dual $S^*=T \notin (aw_{e}),$ as seen in  Example \ref{ex5}.
\end{rema}

Recall  \cite{berkani-zariouh0}  that $T\in (gaw)$ if   $\Delta^g(T)=E_{a}(T).$ And it is proved in \cite{berkani-zariouh0} that the class $(gaw)$ is a  subclass of the class $(aw).$      Our  next proposition, proves  that the class $(aw_{e})$ [resp., $(gaw_{e})$]  is included in  the class $(aw)$ [resp., $(gaw)$]. Moreover,  the left shift operator $L$  shows that in general  these two   inclusions are proper.  We have  $\sigma(L)=\sigma_{a}(T)=\sigma_{bw}(L)=\sigma_{w}(L)=D(0, 1),$ $\sigma_{bf}(L)=\sigma_{e}(L)=C(0, 1)$ and $E_{a}(L)=E_{a}^0(L)=\emptyset.$ Thus $L\in (gaw)$ and then $L\in (aw).$ But $L\not\in (aw_{e})$ and then $L\not\in (gaw_{e}).$

\begin{prop}\label{prop2} Let  $T\in L(X).$ The following statements hold.\\
(i) $T\in (aw_{e})$ if and only if $T\in (aw)$ and $\sigma_{e}(T)=\sigma_{w}(T).$\\
(ii) $T\in (gaw_{e})$ if and only if $T\in (gaw)$ and  $\sigma_{e}(T)=\sigma_{w}(T).$
\end{prop}

\begin{proof} (i) If $T \in (aw_{e})$ then  from Proposition \ref{prop1}, $T \in (ab_{e}).$ From Proposition \ref{cor1}, $\sigma_{e}(T)=\sigma_{w}(T)$ and so  $T\in (aw).$  The converse is obvious.\\
(ii) Goes similarly with (i).
\end{proof}

We give in the following proposition, the relationship between the class    $(aw_{e})$ [resp., $(gaw_{e})$] and  the class $(W_{e})$ [resp., $(gW_{e})$]. The class $(gW_{e})$ was  introduced and studied in \cite{aznay-zariouh} as follows: an operator   $T\in(gW_{e})$ if $\Delta_{e}^g(T)=E(T),$ and it is proved in \cite{aznay-zariouh} that $(gW_{e})$ is a  subclass of $(W_{e}).$
\begin{prop}\label{prop3} Let  $T\in L(X).$ The following statements hold.\\
(i) $T\in (aw_{e})$ if and only if $T\in (W_{e})$ and $E^0(T)=E_{a}^0(T).$\\
(ii) $T\in (gaw_{e})$ if and only if $T\in (gW_{e})$ and  $E(T)=E_{a}(T).$
\end{prop}
\begin{proof}
(i) If   $T\in (aw_{e}),$ then $T\in (ab_{e}).$ 
 From Remark \ref{rema0}  we conclude that  $\Delta_{e}(T)=\Pi^0(T)=\Pi_{a}^0(T)=E_{a}^0(T).$ So $\Delta_{e}(T)=\Pi^0(T)=\Pi_{a}^0(T)=E_{a}^0(T)=E^0(T).$   Thus $T\in (W_{e})$ and $E^0(T)=E_{a}^0(T).$ The converse is clear.\\
 (ii) Goes similarly with (i).
\end{proof}

 From Propositions \ref{prop2} and  \ref{prop3}, we deduce that $(aw_{e})=(aw)\cap(W_{e})$ and $(gaw_{e})=(gaw)\cap(gW_{e}).$   But   the operator $A$ given in Remark \ref{newrema0}, shows that the inclusions  $(aw_{e})\subset (W_{e}) $ and $(gaw_{e})\subset (gW_{e})$   in general are proper.
\par Now, we give a similar results to  Lemma \ref{newlem0} and to Theorem \ref{lem0}. The proof of Theorem \ref{newthm0}  is left to the reader.
\begin{lem}\label{newlema1}  Let $T \in L(X).$    \\
(i) If $E_{a}(T)\subset\Delta_{e}^g(T)$ then $E_{a}^0(T)\subset\Delta_{e}(T).$\\
(ii) If $E_{a}(T)\subset\Delta^g(T)$ then $E_{a}^0(T)\subset\Delta(T).$
\end{lem}
\begin{proof}
(i) Since $E_{a}(T)\subset\Delta_{e}^g(T),$ then     $E_{a}^0(T)\subset E_{a}^0(T)\cap \Delta_{e}^g(T).$  We have always $ E_{a}^0(T)\cap \Delta_{e}^g(T)\subset \Delta_{e}(T).$ Indeed, if $\lambda \in E_{a}^0(T)\cap \Delta_{e}^g(T),$ then $T-\lambda I$ is a B-Fredholm operator with $\alpha(T-\lambda I)<\infty.$ Hence,  $T-\lambda I$ is a semi-Fredholm and consequently  $T-\lambda I$ is a Fredholm operator. Thus $\lambda \in \Delta_{e}(T).$\\
 (ii) Goes similarly with (i).
\end{proof}

\begin{thm}\label{newthm0} Let $T \in L(X).$  The following statements hold.\\
(i) If   $E_{a}^0(T)\subset\Delta_{e}(T)\subset\sigma_{a}(T),$  then $E^0(T)=E_{a}^0(T)=\Pi^0(T)=\Pi_{a}^0(T).$\\
(ii) If   $E_{a}(T)\subset\Delta_{e}^g(T)\subset\sigma_{a}(T),$  then $E(T)=E_{a}(T)=\Pi(T)=\Pi_{a}(T)$ and $E^0(T)=E_{a} ^0(T)=\Pi^0(T)=\Pi_{a}^0(T).$
\end{thm}

\begin{rema}\label{rema3} From   Proposition \ref{prop1}, Proposition \ref{prop3} and Remark \ref{rema0} (or from Theorem \ref{newthm0}  and Proposition \ref{prop2}),  we deduce that   if $T\in  (aw_{e})$ then $\Delta(T)=\Delta_e(T)=\Pi^0(T)=\Pi_a^0(T)=E^0(T)=E_{a}^0(T)$ and $\Delta^g(T)=\Delta_e^g(T)=\Pi(T).$   If in addition,   $T\in (gaw_{e})$ then $\Delta(T)=\Delta_e(T)=\Pi^0(T)=\Pi_a^0(T)=E^0(T)=E_{a}^0(T)$ and $\Delta^g(T)=\Delta_e^g(T)=\Pi(T)=\Pi_a(T)=E(T)=E_{a}(T).$
\end{rema}


As  conclusion, we  give a summary of the results obtained in the two preceding parts of
this paper. In the following diagram,
   the arrows signify the relation of inclusion between
   the known  classes  $(B),$ $(gB),$ $(B_{e}),$ $(gB_{e}),$ $(ab),$  $(gab),$ $(W),$ $(gW),$  $(W_{e}),$ $(gW_{e}),$  $(aw),$  $(gaw)$ and various classes introduced and studied in this paper.
 The numbers near the arrows are references to
the results obtained in the present paper (numbers without brackets) or to
the bibliography therein (the numbers in square brackets).

\begin{tikzpicture}
    \node[state] (a) at (0,0) {$(B)$}; 


    \node[state] (b) at (-3,0) {$(ab_{e})$};
    \node[state] (c) at (-6,0) {$(gab_{e})$};
    \node[state] (d) at (0,-2) {$(gB)$};
    \node[state] (e) at (0,-4) {$(B_{e})$};
    \node[state] (f) at (0,-6)  {$(gB_{e})$};

    \node[state] (g) at (0,2) {$(ab)$};
    \node[state] (h) at (0,4) {$(gab)$};
    \node[state] (i) at (0,6) {$(gaw_{e})$};

    \node[state] (j) at (3,0) {$(W)$};
    \node[state] (k) at (6,0) {$(gW)$};

    \node[state] (n) at (1.5,3.0) {$(aw_{e})$};
    \node[state] (o) at (4,1.5) {$(aw)$};
    \node[state] (p) at (5,3.5) {$(gaw)$};

    \node[state] (q) at (4,-1.5) {$(W_{e})$};
    \node[state] (r) at (5,-3.5) {$(gW_{e})$};

     \path (c) edge node {\scriptsize\ref{cor2}} (b);
     \path (i) edge node[midway, right] {\scriptsize\ref{prop1}} (c);

     \path (h) edge node[midway, right] {\scriptsize\cite{berkani-zariouh0}} (g);
     \path (g) edge node[midway, right] {\scriptsize\cite{berkani-zariouh0}} (a);
     \path (a) edge node[midway, right] {\scriptsize\cite{amouch-zguitti}} (d);
     \path (d) edge (a);
     \path (e) edge node[midway, right] {\scriptsize\cite{aznay-zariouh}} (d);
     \path (e) edge node[midway, right] {\scriptsize\cite{aznay-zariouh}} (f);
     \path (f) edge (e);

     \path (c) edge node[midway, right] {\scriptsize\ref{cor1}} (h);

     \path (b) edge node[midway, right] {\scriptsize\ref{cor1}} (g);
     \path (i) edge  node[midway, right] {\scriptsize\ref{cor3}} (n);

     \path (n) edge node[midway, right] {\scriptsize\ref{prop2}} (o);
     \path (i) edge node[midway, right] {\scriptsize\ref{prop2}} (p) ;
     \path (p) edge node[midway, right] {\scriptsize\cite{berkani-zariouh0}} (o);
     \path (o) edge node[midway, right] {\scriptsize\cite{berkani-zariouh0}} (j);
     \path (p) edge node[midway, right] {\scriptsize\cite{berkani-zariouh0}} (k);
     \path (k) edge node[midway] {\scriptsize\cite{Berkani-koliha}} (j);
     \path (j) edge node[midway] {\scriptsize\cite{harte-lee}}  (a);

     \path (q) edge node[midway, right] {\scriptsize\cite{aznay-zariouh}} (j);
     \path (r) edge node[midway, right] {\scriptsize\cite{aznay-zariouh}} (q);
     \path (q) edge node[midway, right] {\scriptsize\cite{aznay-zariouh}} (e);
     \path (r) edge  node[midway, right] {\scriptsize\cite{aznay-zariouh}} (k);
     \path (o) edge  node[midway, right] {\scriptsize\cite{berkani-zariouh0}} (g);
     \path (p) edge node[midway, right] {\scriptsize\cite{berkani-zariouh0}} (h);
     \path (b) edge node[midway, right] {\scriptsize\ref{cor0}} (e);

     \path (n) edge[bend right=35] node[midway, right] {\scriptsize\ref{prop3}} (q);
     \path (i) edge[bend left=95] node[midway, right] {\scriptsize\ref{prop3}} (r);
     \path  (n) edge[bend right=40] node[midway, right] {\scriptsize\ref{prop1}} (b);

\end{tikzpicture}

\goodbreak
{\small \noindent Zakariae Aznay,\\  Laboratory (L.A.N.O), Department of Mathematics,\\Faculty of Science, Mohammed I University,\\  Oujda 60000 Morocco.\\
aznay.zakariae@ump.ac.ma\\

\noindent Abdelmalek Ouahab,\newline Laboratory (L.A.N.O), Department of
	Mathematics,\newline Faculty of Science, Mohammed I University,\\
	\noindent Oujda 60000 Morocco.\\
	\noindent ouahab05@yahoo.fr\\

 \noindent Hassan  Zariouh,\newline Department of
Mathematics (CRMEFO),\newline
 \noindent and laboratory (L.A.N.O), Faculty of Science,\newline
  Mohammed I University, Oujda 60000 Morocco.\\
 \noindent h.zariouh@yahoo.fr

\end{document}